\documentclass[12pt,letterpaper]{amsart}
\usepackage{geometry}
\usepackage{mathrsfs,amsmath,amssymb,amsthm,amsfonts}
\usepackage{latexsym,amscd,setspace,color}
\usepackage{mathpazo} 
\usepackage{hyperref}
\usepackage{color}

\newtheorem{Theorem}{Theorem}[section]
\newtheorem{Definition}[Theorem]{Definition}
\newtheorem{Proposition}[Theorem]{Proposition}

\newtheorem{Lemma}[Theorem]{Lemma}
\newtheorem{Corollary}[Theorem]{Corollary}
\theoremstyle{remark}
\newtheorem{Example}[Theorem]{Example}

\def\ovr{\overline}

\def\om{\omega}
\def\Om{\Omega}
\def\al{\alpha}

\def\ve{\varepsilon}

\def\lm{\lambda}

\def\sbs{\subset}

\def\lim{\operatorname{lim}}

\def\be{\begin{enumerate}}
\def\ee{\end{enumerate}}
\def\bT{\begin{Theorem}}
\def\eT{\end{Theorem}}
\def\bP{\begin{Proposition}}
\def\eP{\end{Proposition}}
\def\bPr{\begin{proof}}
\def\ePr{\end{proof}}
\def\bD{\begin{Definition}}
\def\eD{\end{Definition}}
\def\bE{\begin{Example}}
\def\eE{\end{Example}}
\def\bL{\begin{Lemma}}
\def\eL{\end{Lemma}}
\def\bC{\begin{Corollary}}
\def\eC{\end{Corollary}}

\begin{document}
\title{Poletsky-Stessin-Hardy spaces in the plane}
\author{Muhammed Al\.{i} Alan}
\address[Muhammed Al\.{i} Alan]{Syracuse University, Syracuse, NY,
13244 USA} \email{malan@syr.edu}
\author{N\.{i}hat G\"{o}khan G\"{o}\u{g}\"{u}\c{s}}
\address[N\.{i}hat G\"{o}khan G\"{o}\u{g}\"{u}\c{s}]{ Sabanci University,  Tuzla , Istanbul 34956 TURKEY}
\email{nggogus@sabanciuniv.edu}
\begin{abstract}
In this paper we give two complete characterizations of the
Poletsky- Stessin- Hardy spaces in the complex plane: First in terms
of their boundary values as a weighted subclass of the usual $L^p$
class with respect to the arclength measure on the boundary. Second
we completely describe functions in these spaces by having a
harmonic majorant with a certain growth condition and we prove some
basic results about these spaces. In particular, we prove
approximation results in such spaces and extend the classical result
of Beurling which describes the invariant subspaces of the shift
operator. Additionally we provide non-trivial examples.
\end{abstract}
\maketitle

\section{Introduction}
The theory of Hardy Spaces started with works of G. H. Hardy, J. E.
Littlewood in 1920s. By works of them and I. I. Privalov, F. and
M. Riesz, V. Smirnov and G. Szeg\"{o} the  theory was developed. It
is now an important branch of function theory still continuing to
attract many people in the common field of operator theory and
complex analysis. Their work was mainly in the unit disk of
$\mathbb{C}$. Later Hardy space theory was extended to more general
classes of domains such as the ball of ${\mathbb{C}}^n$
\cite{Rudin2}, the polydisk \cite{Rudin3}, multiply connected
domains in $\mathbb{C}$, Smirnov domains \cite{Duren}, pseudoconvex
domains with $C^2$ boundaries \cite{EaFefRos}, \cite{Ste72}.

This paper arose from the substantial development of function theory
in the work of E. A. Poletsky and M. Stessin \cite{PolSte08}. Using
the beautiful construction of Demailly-Monge-Amp\'{e}re measures
\cite{Dem87} of J. P. Demailly they unified those theories so that
one does not need to give a separate definition for all different
kinds of domains. Of course what those domains have in common is
that they are all hyperconvex.

We restrict our prime focus to the theory of Poletsky-Stessin-Hardy
spaces in the complex plane. We mention that in complex analysis of
one variable a domain is hyperconvex means exactly the same as that
the domain is regular with respect to the Dirichlet problem. This
means that the new theory of Hardy spaces applies to a wide class of
domains. Besides in one variable the theory is much transparent and
accessible because of the abundance of the tools from potential
theory. Even the unit disk accommodates a rich theory which is not
yet studied extensively. In this paper for example we provide the
first non-trivial examples of Poletsky-Stessin-Hardy spaces.

It turns out that Poletsky-Stessin-Hardy
spaces in the plane can be completely characterized by their boundary values. These spaces are
isometrically isomorphic to a subclass of a weighted $L^p$ space with respect to a positive
measure on the boundary (Theorem \ref{T:CharacHardyClassesBoundaryVal}). This important
property leads to very fruitful discussions and allows to extend the classical results of the theory.

Let us briefly introduce the content of the paper. Following \cite{PolSte08}
we define Poletsky-Stessin-Hardy spaces in this paper on a
regular domain $G$ in $\mathbb{C}$ as the set of all analytic functions on $G$ such that they satisfy the
following integral growth condition:
\begin{eqnarray*}
 \sup_{c<0}\left (\int_{S_{c,u}}|f|^p\,d\mu _{c,u}\right )^{1/p}<\infty
\end{eqnarray*}
(for the details and definitions see section \ref{Sec:Body}). Here $u$ is an arbitrary subharmonic
exhausting function for $G$. If $G=\mathbb{D}$ and $u=\log |z|$, we obtain the classical $H^p$ space
on the disk.

Section \ref{Sec:Body} contains our main theorem (Theorem
\ref{T:CharacHardyClassesBoundaryVal}) which completely
characterizes these spaces by their boundary values and by the
existence of a harmonic majorant having a certain growth condition.
On one hand this result is an extension of the well-known
characterization of functions in the classical Hardy space. On the
other hand, the main theorem reveals the structure of the functions
in the Poletsky-Stessin-Hardy spaces. In particular we learn from
this theorem that the zeros of a function in this space can be
factored in such a way that the resulting function is still in the
same space. Moreover, we show that the disk algebra or the set of
all complex polynomials is dense in this space whenever the Laplace
of the exhaustion has finite mass on the unit disk.

In section \ref{Sec:InvSubspace} we give a complete characterization
of the closed invariant subspaces of the multiplication by the
identity function which can be viewed as the shift operator. This is
a well-known theorem of Beurling in the classical case. Our result
in this sense is a generalization of Beurling's theorem.

In section \ref{Sec:Examples} we provide the first non-trivial
examples of Poletsky-Stessin-Hardy spaces. Finally we conclude the
paper with some remarks.

\section{Hardy spaces in the plane}\label{Sec:Body}

A function $u\leq 0$ on a bounded open set $G\subset \mathbb{C}$ is called
an exhaustion on $G$ if the set
\[
B_{c,u}:=\{z\in G:u(z)< c\}
\]
is relatively compact in $G$ for any $c<0$.
It is known that there is a subharmonic exhaustion function on $G$ if and only if $G$ is regular.
If $u$ is an exhaustion and $c<0$ is a number, we set
\[
u_c:=\max\{u,c\},\,\,\,\,\,\, S_{c,u}:=\{z\in G:u(z)=c\}.
\]
Since $u_c$ is a continuous subharmonic function the measure $\Delta u_c$
is well-defined. We define
\begin{eqnarray*}
 \mu_{c,u}:=\Delta u_c-\chi_{G\backslash B_{c,u}}\Delta u,
\end{eqnarray*}
where $\chi_{\om}$ is the characteristic function of a set $\om\sbs G$.
Demailly \cite{Dem87} calls these measures as Monge-Amp\'{e}re measures. We shall call those measures
as Demailly-Monge-Amp\'{e}re measures.

If $u$ is a negative subharmonic exhaustion function on $G$,
then the Demailly-Lelong-Jensen formula takes the form
\begin{equation}
\int_{S_{c,u}}v\,d\mu _{c,u}=\int_{B_{c,u}}(v\Delta u-u\Delta
v)+c\int_{B_{c,u}}\Delta v,  \label{Eq:DemJenLel}
\end{equation}%
where $\mu _{c,u}$ is the Demailly measure which is supported in the level sets $%
S_{c,u}$ of $u$ and $v\in $ $sh(G)$. This formula was proved by
Demailly \cite{Dem87} for hyperconvex domains in $\mathbb{C}^n$. In
complex analysis of one variable a domain is hyperconvex means
exactly the same as that the domain is regular. Let us recall that by \cite{Dem87}
if $\int_{G}\Delta u<\infty $, then the
measures $\mu _{c,u}$ converge as $c\rightarrow 0$ weak-$\ast $ to a measure $%
\mu _{u}$ supported in the boundary $\partial G$.
\newline

Let $u \in sh(G)$ be an exhaustion function which is continuous with values in $\mathbb{R}\cup \{-\infty\}$.
Following \cite{PolSte08} we set
\[
\mathnormal{sh}_u(G):=\mathnormal{sh}_u:=\left\{v\in sh(G):v\geq
0,\,\sup _{c<0}\int_{S_{c,u}}v\,d\mu _{c,u}<\infty\right\},
\]%
and
\[
H^p_u(G):=H^p_u:=\left\{f\in hol(G):|f|^p\in
\mathnormal{sh}_u\right\}
\]
for every $p>0$. In hyperconvex domains in $\mathbb{C}^n$ using
only the pluricomplex Green function for the domain the class
of holomorphic Hardy spaces was first introduced in \cite{Al03}. These spaces were
independently introduced and extensively studied in \cite{PolSte08}. We write
\begin{eqnarray}
\|v\|_{u}:=\sup
_{c<0}\int_{S_{c,u}}v\,d\mu _{c,u}=\int_{G}(v\Delta u-u\Delta v)
\end{eqnarray}
for the norm of a nonnegative function $v\in sh (G)$ and set
\begin{eqnarray}
\|f\|_{u,p}:=\sup_{c<0}\left (\int_{S_{c,u}}|f|^p\,d\mu _{c,u}\right )^{1/p}
\end{eqnarray}
for the norm of a holomorphic function $f$ on $G$.
Let us write $\|f\|_{u}$ when $p=1$.
It is known in view of \cite[Theorem 4.1]{PolSte08} that $H^p_u$ is
a Banach space when $p\geq 1$.  It is clear that the
function $1$ belongs to $H^p_u$ if and only if the Demailly measure
$\mu_{u}$ has finite mass. If $G$ is a regular bounded domain in
$\mathbb{C}$ and $w\in G$, then we have the Green function
$v(z)=g_G(z,w)$ which is a subharmonic exhaustion function for $G$.

\par Let $G$ be a domain of class $C^2$. This means that there
exists a real function $\lm$ (which is so called the characterizing
function) defined in a neighborhood of $\ovr G$ so that
\begin{itemize}
\item[i.] $G=\{z\in
\mathbb{C}
:\lambda (z)<0\}$, and
\item[ii.] $|\nabla \lambda (z)|>0$ if $z\in\partial G$.
\end{itemize}
Let \[G_{\varepsilon }:=\{z\in
\mathbb{C}
:\lambda (z)<-\varepsilon \},\,\,\,\, \ve >0.\]We will denote by
$\nu$ the usual arclength measure on $\partial G$ normalized so that
$\nu(\partial G)=1$. Let $\nu_{\varepsilon }$ denote the normalized
arclength measure on $G_{\varepsilon }. $ The following theorem was
proved in \cite[Theorem 1]{Ste72} (see also the Corollary after
that and \cite[Chapter 10]{Duren}).

\bT\label{T:SteinSpace} Suppose $u$ is a real harmonic function in a
bounded $C^2$ domain $G$. Then the following properties are
equivalent:
\begin{itemize}
\item[i.] $\underset{\varepsilon >0}{\sup }\left( \int_{\partial G_{\varepsilon
}}|u(\zeta )|^{p}d\nu _{\varepsilon }(\zeta )\right)
^{1/p}<\infty,\,\, $ $1\leq p$.

\item[ii.] $u(z)=\int_{\partial G}\widetilde{u}(\zeta)P_G(z,\zeta)d\nu (\zeta )$, where $\widetilde{u}\in L^{p}(\partial G)$ if $p>1$;
when $p=1$, then $\widetilde{u}(\zeta)d\nu (\zeta )$ has to be
replaced by a finite Borel measure on $\partial G$.

\item[iii.] $|u(z)|^p$ has a harmonic majorant if $p<\infty$. When
$p=\infty$, then we assume that $u$ is bounded in $G$.
\end{itemize}

Also
\[\left\Vert \widetilde{u}\right\Vert _{p}\leq \underset{\varepsilon >0}{\sup }\left(
\int_{\partial G_{\varepsilon }}(u(\zeta ))^{p}d\nu _{\varepsilon
}(\zeta )\right) ^{1/p},\,\,\,\, \textnormal{if} \,\,\,\, p>1.\] \eT

The following potential theory result is quite useful (see for example
\cite[Theorem 4.5.4]{Ran}).
\bL\label{Lem:HarmMajorant}
Let $G$ be a bounded domain of $\mathbb{C}$
and $u$ be a subharmonic function on $G$ which is not identically $-\infty$.
\begin{itemize}
 \item[i.] If $u$ has a harmonic majorant on $G$, then it has a least harmonic majorant $h$
 and we have
 \[u(z)=h(z)-\frac{1}{2\pi}\int_G g_G(z,w)\Delta u(w), \,\,z\in G.\]
 \item[ii.] If $u$ has no majorant on $G$, then
 \[-\frac{1}{2\pi}\int_G g_G(z,w)\Delta u(w)=\infty, \,\,z\in G.\]
\end{itemize}
\eL

We present now a complete description of the functions in
Poletsky-Stessin-Hardy spaces by the existence of a harmonic
majorant.

\bT \label{T:CharacHardyClassesHarmonicMajorant} Let $G$ be a
bounded domain and $u$ be a subharmonic exhaustion function on $G$.
Let $p>0$. The following statements are equivalent:
\begin{itemize}
\item[i.] $f\in H^p_u(G)$.
\item[ii.] There exists a least harmonic function $h$ in $G$ which
belongs to the class $sh_u$ so that $|f|^p\leq h$ on $G$.
Furthermore,
\begin{eqnarray*}
\|f\|_{u,p}^p=\int _G h\Delta u=\|h\|_u.
\end{eqnarray*}
\end{itemize}
\eT

\begin{proof}
Using (\ref{Eq:DemJenLel}) and Poisson-Jensen formula we see that
\begin{eqnarray}
\|f\|_{u,p} ^p&=& \nonumber \int_{G}|f(z)|^p\Delta u(z)- \int_{G}
u(z) \Delta (|f(z)|^p)  \\ \nonumber &=& \int_{G}|f(z)|^p\Delta
u(z)- \int_{G} \left
[\int_{G}g_G(z,w)\Delta u(w) \right ] \Delta(|f(z)|^p) \\
\label{Eq:har} &=&   \int_{G} \left [
|f(z)|^p-\int_Gg_G(z,w)\Delta(|f(w)|^p)\right ] \Delta u(z)
\end{eqnarray} We let
\[h(z):=|f(z)|^p-\int_Gg_G(z,w)\Delta(|f(w)|^p).\]
From Lemma \ref{Lem:HarmMajorant} we see that the subharmonic function $|f|^p$ has a least harmonic
majorant if and only if $h(z)<\infty$ for $z\in G$ and
in this case $h$ is the least harmonic majorant. Note that by (\ref{Eq:har}) $\|f\|_{u,p}
<\infty$ if and only if $h\in sh_u$. Therefore i. and ii. are
equivalent.
\end{proof}

\bC Let $G$ be a bounded regular domain, $w\in G$, and
$u=g_{G}(z,w)$ be a Green function for $G$. A holomorphic function
$f$ in $G$ belongs to $H^p_u$ if and only if $|f|^p$ has a harmonic
majorant in $G$. \eC

\begin{proof}
Suppose $|f|^p$ has a harmonic majorant $h$. Then
\[\|f\|_{u,p}^p=\sup_{c<0}\int_{S_{c,u}}|f|^p\,d\mu _{c,u}\leq \sup_{c<0}\int_{B_{c,u}}h\,\Delta (g_{G}(z,w))=h(w)<\infty.\]
Conversely, suppose $f\in H^p_u(G)$. Since the Green functions
$-g_{B_{c,u}}(z,w)$ are increasing to $-g_{G}(z,w)$ as $c$ increases
to $0$, we have
\begin{eqnarray*}
|f(w)|^p &=& \sup_{c<0}
\left(\frac{1}{2\pi}\int_{S_{c,u}}|f|^p\,d\mu _{c,u}
-\frac{1}{2\pi}\int_{B_{c,u}}g_{B_{c,u}}(z,w)\,\Delta
(|f(z)|^p)\right) \nonumber \\ &=&
\|f\|_{u,p}^p-\frac{1}{2\pi}\int_{G}g_{G}(z,w)\,\Delta
(|f(z)|^p)<\infty.
\end{eqnarray*}
By Lemma \ref{Lem:HarmMajorant} $|f|^p$ has a harmonic majorant on
$G$.
\end{proof}

\par Let $G$ be a domain in $\mathbb{C}$ with $C^2$ boundary. Let
$u(z)=g_G(z,w)$ be the Green function and $\lm$ be a characterizing
function for $G$. We get the Hardy space $H^p_u$ in the sense of
Poletsky-Stessin and also we have the Hardy space $H^p_{\lm}$ in the
usual sense (for example see \cite{Duren} and \cite{Ste72}). 
So a holomorphic function $f$ on $G$ belongs to
$H^p_{\lm}$ if and only if
\[\underset{\varepsilon >0}{\sup }\left( \int_{\partial G_{\varepsilon
}}|f(\zeta )|^{p}d\nu _{\varepsilon }(\zeta )\right)
^{1/p}<\infty.\] Theorem \ref{T:SteinSpace} says that a holomorphic
function $f$ on a $C^2$ domain $G$ belongs to the class $H^p_{\lm}$
if and only if $|f|^p$ has a harmonic majorant. Given these facts we
arrive at the following:

\bC\label{C:EquivalenceHardyStein} If $G$ is a domain in
$\mathbb{C}$ with $C^2$ boundary and $u(z)=g(z,w)$ is a Green
function for $G$, then $H^p_{\lm}=H^p_u$ for $p\geq 1$.
\eC

 Let $G$ be a bounded domain in $\mathbb{C}$ with $C^2$
boundary. Following Stein for each $\zeta\in \partial G$, we denote
by $\eta _{\zeta}$ the unit outward normal at $\zeta$. For each
$\al>0$ we define an approach region $\Lambda _{\alpha }(\zeta )$
with vertex $\zeta$ by the equation
\[\Lambda _{\alpha }(\zeta ):=\{z\in G:|(z-\zeta )\cdot \ovr{\eta} _{\zeta}|<(1+\alpha )\delta _{\zeta
}(z),|z-\zeta |^{2}<\alpha \delta _{\zeta }(z)\}.\] Here $\delta
_{\zeta }(z)$ is the distance from $z$ to $\partial G$.  The
following observation is important:

\bT\cite[Theorem 10]{Ste72}\label{T:SteinBoundaryValues} Suppose $f$
belongs to $H^p_{\lm}$, $0<p<\infty$. Then $f$ has admissible (and
non-tangential) limits at almost every boundary point and
\[\int_{\partial G}\underset{z\in \Lambda _{\alpha }(\zeta )}{\sup
}|f(z
)|^{p}d\nu (\zeta )\leq C_{p,\alpha }\underset{\varepsilon >0}{\sup }%
\int_{\partial G_{\varepsilon }}|f(\zeta )|^{p}d\nu _{\varepsilon
}(\zeta ).\] \eT

We will denote the non-tangential limits of $f$ by $f^{\ast}$. The
following result shows that the harmonic function on $G$ with values
$|f^{\ast}|^p$ coincides with the least harmonic majorant of
$|f|^p$.

\begin{Theorem}\label{T:BoundaryValues} Let $G$
be a bounded domain with $C^2$ boundary. Let $u$ be a real-valued
harmonic function on $G$ so that
\[\underset{\varepsilon >0}{\sup }\left( \int_{\partial G_{\varepsilon
}}|u(\zeta )|^{p}d\nu _{\varepsilon }(\zeta )\right)
^{1/p}<\infty.\] If $p>1$, then the boundary values $\widetilde{u}$
coincides with the non-tangential boundary values $u^{\ast}$.
Moreover, if $h$ is the least harmonic majorant of $|u|^p$, then
\begin{eqnarray}\label{EqLeastHMajEqualsPoissonInt}
h(z)=\int_{\partial G}|u^{\ast}(\zeta)|^pP_G(z,\zeta)d\nu (\zeta )
\end{eqnarray} for every $z\in G$.
\end{Theorem}

\begin{proof} From Theorem \ref{T:SteinSpace}
\[u(z)=\int_{\partial G}\widetilde{u}(\zeta)P_G(z,\zeta)d\nu (\zeta )\]
on $G$. Then the classical Perron-Wiener-Brelot method (see
\cite[Section 2]{Helms}) implies that the non-tangential limits
$u^{\ast}$ of $u$ converges for $\nu$-almost every $\zeta$ in
$\partial G$ to $\widetilde{u}(\zeta)$ (see also \cite[Theorem
4]{Ste72}). In fact non-tangential limit of $u$ exists at every
boundary point which is in the Lebesgue set of $\widetilde{u}$. Let
$G_{\ve}$ be approximating regions in $G$ and $\pi_{\ve}:\partial
G_{\ve}\to \partial G$ be the normal projection. Define
\[\widetilde{u}_{\ve}(\zeta):=u(\pi_{\ve}^{-1}(\zeta)), \,\,\zeta\in\partial G.\]
From the construction (see the proof of Theorem 1 in \cite{Ste72})
these functions converge in $L^p(d\nu)$ to $\widetilde{u}$ as
$\ve\to 0$. Let $h_{\ve}$ be the least harmonic majorant of $|u|^p$
on $G_{\ve}$. By the Poisson-Jensen formula
\begin{eqnarray}
h_{\ve}(z) &=& \label{Eq:LeastHarmMaj}\int_{\partial G_{\ve}}
|u(s)|^pP_{G_{\ve}}(z,s)\,d\nu_{\ve}(s) \\ \nonumber &=&
\int_{\partial G}
|\widetilde{u}_{\ve}(\zeta)|^pP_{G}(z,\pi_{\ve}^{-1}(\zeta)){\mathcal{J}}_{\ve}(\pi_{\ve}^{-1}(\zeta))\,d\nu(\zeta),
\end{eqnarray}
where $z\in G$ and ${\mathcal{J}}_{\ve}$ is the Jacobian for
$\pi_{\ve}$. Fix a point $z_0\in G$. Since $G$ has $C^2$ boundary,
the Poisson kernels $P_{G}(z_0,\pi_{\ve}^{-1}(\zeta))$ converge
uniformly in $\zeta$ to $P_{G}(z_0,\zeta)$. Also the Jacobians
${\mathcal{J}}_{\ve}(\pi_{\ve}^{-1}(\zeta))$ converge uniformly in
$\zeta$ to $1$ as $\ve\to 0$. Using then the fact that
$\widetilde{u}_{\ve}$ converge to $\widetilde{u}$ in $L^p(d\nu)$ we
get that the last integrals in (\ref{Eq:LeastHarmMaj}) converge to
the integral
\[\int_{\partial G}
|\widetilde{u}(\zeta)|^pP_{G}(z,\zeta)\,d\nu(\zeta).\] On the other
hand, $h_{\ve}(z)$ converge to $h(z)$ for every $z\in G$. This
proves the equality in (\ref{EqLeastHMajEqualsPoissonInt}).
\end{proof}

Theorem \ref{T:BoundaryValues} can easily be adopted to get the
following Corollary:

\begin{Corollary}\label{Cor:HarmMajRecoveredBndryValues}
Let $G$ be a bounded domain with $C^2$ boundary. Let $f\in
H^p_{\lm}$, $p>1$. If $h$ is the least harmonic majorant of $|f|^p$
and $f^{\ast}$ denotes the non-tangential limits of $f$, then
\begin{eqnarray}\label{Eq:LeastHMajEqualsPoissonIntHolCase}
h(z)=\int_{\partial G}|f^{\ast}(\zeta)|^pP_G(z,\zeta)d\nu (\zeta )
\end{eqnarray} for every $z\in G$.
\end{Corollary}

We will borrow a result from \cite{PolSte08}. \bP\cite[Corollary
3.2]{PolSte08} \label{P:HardyClassesContainedGreenExh} Let $u$ be a
continuous subharmonic exhaustion function on a bounded regular
domain in $G$ and let $v(z)=g(z,w)$ be the Green function. Then
$sh^p_u(G)\sbs sh^p_v(G)$ and there is a constant $c$ such that
$\|\varphi\|_{v}\leq c\|\varphi\|_{u}$ for every nonnegative
subharmonic function $\varphi$ on $G$. In particular, $H^p_u(G)\sbs
H^p_v(G)$. \eP

Our main theorem says that Poletsky-Stessin-Hardy spaces are
isometric to a subspace of some $L^p(\widetilde{\mu_u})$. Hence
these spaces are a generalization of the classical theory of Hardy
spaces to a weighted theory. \bT
\label{T:CharacHardyClassesBoundaryVal} Let $G$ be a bounded domain
with $C^2$ boundary and $u$ be a subharmonic exhaustion function on
$G$. Let $p>1$. The following statements are equivalent:
\begin{itemize}
\item[i.] $f\in H^p_u(G)$.
\item[ii.] $f\in H^p_{\lm}(G)$ and \[\int_{\partial G} |f^{\ast
}(\zeta )|^pV(\zeta )d\zeta <\infty ,\] where
\begin{eqnarray}\label{Eq:DefnV} V(\zeta):=\int_{G}P_G(z,\zeta)\Delta u(z),\,\,\,\,
\zeta\in\partial G.\end{eqnarray}
\item[iii.] $f\in H^p_{\lm}(G)$ and there exists a positive measure
$\widetilde{\mu_u}$ on $\partial G$ such that $|f^*|\in
L^p(\widetilde{\mu_u})$. Moreover, if $E$ is any Borel subset of
$\partial G$ with measure $\nu(E)=0$, then $\widetilde{\mu_u}(E)=0$
and we have the equality
\begin{eqnarray} \int_{\partial G}\varphi\,d\widetilde{\mu_u}=\int_{G}P_G(\varphi)\Delta u\end{eqnarray}
for every $\varphi\in L^1(\partial G)$.
\end{itemize}
In addition, if $f\in H^p_u(G)$, then
$\|f\|_{u,p}=\|f^{\ast}\|_{L^p(\widetilde{\mu_u})}$. \eT

\begin{proof}
Let $V(\zeta)$ be the non-negative function defined in
(\ref{Eq:DefnV}). We set
\begin{eqnarray*}
d\widetilde{\mu_u}(\zeta):=V(\zeta)d\nu(\zeta).
\end{eqnarray*}
Let
\begin{eqnarray*}
H(z):=\int_{\partial G} |f^{\ast }(\zeta
)|^pP_G(z,\zeta)d\nu(\zeta),\,\,\,\,z\in G,
\end{eqnarray*}
and let $h$ denote the least harmonic majorant of $|f|^p$ whenever
exists. We take $H\equiv \infty$ when
$\|f^{\ast}\|_{L^p(\widetilde{\mu_u})}=\infty$ and $h\equiv \infty$
when $\|f\|_{u,p}=\infty$. Using Fubini's Theorem gives
\begin{eqnarray}
\|f^{\ast}\|_{L^p(\widetilde{\mu_u})}^p &=& \int_{\partial G}
|f^{\ast }(\zeta )|^pV(\zeta )d\nu(\zeta) \nonumber \\ &=&
\int_{\partial G} |f^{\ast }(\zeta )|^p\left [ \int_G
P_G(z,\zeta)\Delta u(z)\right ]d\nu(\zeta) \nonumber \\ &=& \int_{G}
H(z) \Delta u(z). \label{Eq:BndryHol}
\end{eqnarray}

From Theorem \ref{T:CharacHardyClassesHarmonicMajorant} we had
\begin{eqnarray}
 \|f\|_{u,p}^p  = \int_{G} h(z)\Delta u(z) \label{Eq:hol}.
\end{eqnarray}

If $f\in H^p_u(G)$, then $f\in H^p_{\lm}$ by Proposition
\ref{P:HardyClassesContainedGreenExh} and Corollary
\ref{C:EquivalenceHardyStein}. Using Theorem \ref{T:BoundaryValues}
we see that there is equality of (\ref{Eq:BndryHol}) and
(\ref{Eq:hol}). Hence we have $ \|f\|_{u,p} =
\|f^{\ast}\|_{L^p(\widetilde{\mu_u})}$ and i. implies ii.
Conversely, if ii. holds, then in view of Corollary
\ref{Cor:HarmMajRecoveredBndryValues} the least harmonic majorant
$h$ of $|f|^p$ is finite, and $H=h$ on $G$. Again
(\ref{Eq:BndryHol}) and (\ref{Eq:hol}) coincide, hence $f\in H^p_u$.
Thus i. and ii. are equivalent.

Clearly the measure $\widetilde{\mu_u}$ satisfies the conditions in
iii. From the definition of $\widetilde{\mu_u}$ we have
\begin{eqnarray*}
\int _{\partial G}|f^{\ast }(\zeta )|^p\,d\widetilde{\mu_u}(\zeta) =
\int_{\partial G} |f^{\ast }(\zeta )|^pV(\zeta )d\zeta.
\end{eqnarray*}
We see that $|f^{\ast }|\in L^p(\widetilde{\mu_u})$ if and only if
ii. holds. This proves that ii. and iii. are equivalent and
completes the proof of the theorem.
\end{proof}
In particular, when the function $u$ is chosen to be $\log|z|$ on
the domain $G=\mathbb{D}$, the weight function $V(\zeta)$ becomes
identically equal to $1$ giving the classical Hardy space
$H^p(\mathbb{D})$. One of the conclusions of Theorem
\ref{T:CharacHardyClassesBoundaryVal} is the following analog of the
Riesz factorization result. \bT\label{Th:InnerOuterFactorization}
Let $f=gh$ for some $g\in H^{\infty}(G)$ and $h\in hol(G)$. Then
$f\in H^p_u$ if $h\in H^p_u$, $p>0$. Suppose $G$ has $C^2$ boundary,
$p>1$ and $|g^*(\zeta)|$ is $1$ for $\nu$-almost every $\zeta$ on
$\partial G$. Then $h\in H^p_u$ if and only if   $f\in H^p_u$. In
fact, $\|f\|_{p,u}=\|h\|_{p,u}$ in this case. \eT

Our next aim is to show that Poletsky-Stessin-Hardy spaces are
carried the same by conformal maps between bounded domains.

\bT Let $U$ and $G$ be bounded regular domains in $\mathbb{C}$ and
$\varphi :\ovr{G}\to\ovr{U}$ be a $C^1$ function with $C^1$ inverse
which is conformal from $G$ onto $U$. Let $u$ be a subharmonic
exhaustion function on $U$. Then
\begin{eqnarray*}
f\in H^p_u(U)  \,\,\,\,\,\,\textnormal{if and only if}\,\,\,\,\,\, f\circ \varphi\in H^p_{u\circ\varphi}(G).
\end{eqnarray*}
\eT
\begin{proof}
Since $\varphi$ is conformal, there are constants $m$ and $M$ so
that $0<m\leq |\varphi'|^2\leq M$ on $G$ and since $\varphi$ is
continuous on $\ovr G$ the function $u\circ \varphi$ is also a
subharmonic exhaustion on $G$. Also $m\Delta u\leq\Delta (u\circ
\varphi)=\Delta u|\varphi'|^2\leq M\Delta u$. Let $f\in H^p_u(U)$.
Considering Theorem \ref{T:CharacHardyClassesHarmonicMajorant} this
means that there exists $h\in har(U)$ so that $\int h\Delta
u<\infty$ and $|f|^p\leq h$ on $U$. In this case $h\circ\varphi\in
har(G)$, $\int (h\circ\varphi)\Delta (u\circ\varphi)<\infty$, and
$|f\circ\varphi|^p\leq h\circ\varphi$ on $G$. Using Theorem
\ref{T:CharacHardyClassesBoundaryVal} we see that $ f\circ
\varphi\in H^p_{u\circ\varphi}$. We can apply the same argument for
$\varphi ^{-1}$ to complete the proof.
\end{proof}

\par The following observation is elementary and useful.
\bL \label{Lem:HarmExt} Let $G$ be a regular bounded domain in $\mathbb{C}$ and $u\in
C(\ovr G)\cap sh (G)$ be an exhaustion. Then
\begin{eqnarray*}
\int_{S_{c,u}}\varphi (\zeta )d\mu _{c,u}(\zeta )
=\int_{B_{c,u}}P_{B_{c,u}}\varphi\Delta u(z)\end{eqnarray*} for
every $\varphi\in C(\ovr G)$. \eL

\begin{proof}
 Let $\varphi$ be a continuous function on $S_{c,u}$ and let $h(z)$ be the harmonic function in $B_{c,u}$ with boundary
values equal to $\varphi$. We know that $u-c$ is a subharmonic
exhaustion function for $B_{c,u}$. By equality (\ref{Eq:DemJenLel}) we have
\begin{eqnarray*}
\int_{S_{c,u}}\varphi (\zeta )d\mu _{c,u}(\zeta ) = \int_{B_{c,u}}h(z)\Delta u(z).
\end{eqnarray*}
\end{proof}

\par To obtain Fatou's type results we would like to compute the
Radon-Nikodym derivative of the Demailly measures with respect to
the usual arclength measure on the level sets. In the next result we
provide this. Suppose that $S_{c,u}$ is a rectifiable Jordan curve.
Let $\nu_c$ denote the arclength measure on $S_{c,u}$. Suppose that
the harmonic measure $d\om (z,\zeta)$ for $B_{c,u}$ can be expressed
as
\begin{eqnarray*}
d\om (z,\zeta)=P_{B_{c,u}}(z,\zeta)d\nu_c(\zeta)
\end{eqnarray*}
and define
\begin{eqnarray*}U_c(\zeta):=\int_{B_{c,u}}P_{B_{c,u}}(z,\zeta)\Delta
u(z),\,\,\,\,\zeta\in S_{c,u},
\end{eqnarray*} where $P_{B_{c,u}}(z,\zeta)$ denotes the Poisson kernel for
$B_{c,u}$. This is the case for example when $B_{c,u}$ is chord-arc
(see for example \cite{CaKeLa} and Theorem 1.14 therein) and this
happens for instance when $S_{c,u}$ is Lipschitz (see also \cite{GarMar}).

\bP\label{Prop:AbsContLebesgueMeas} Let $u$ be a subharmonic exhaustion function on a bounded
regular domain $G$ in $\mathbb{C}$. Suppose that $u$ is Lipschitz in
every compact subset of $G$. Then the measures $\nu_c$ and
$\mu_{c,u}$ are mutually absolutely continuous and $\mu
_{c,u}=U_{c}\nu _{c}$ with $U_c\in L^1(\nu_c)$. \eP

\begin{proof} Let $\varphi$ be a continuous function on $S_{c,u}$ and let $h(z)$ be the harmonic function in $B_{c,u}$ with boundary
values equal to $\varphi$. By equality (\ref{Eq:DemJenLel})
and Lemma \ref{Lem:HarmExt} we have
\begin{eqnarray*}
\int_{S_{c,u}}\varphi (\zeta )d\mu _{c,u}(\zeta )
&=&\int_{B_{c,u}}h(z)\Delta u(z) \\
&=&\int_{B_{c,u}}\left( \int_{S_{c,u}}\varphi (\zeta
)P_{B_{c,u}}(z,\zeta
)d\nu _{c}(\zeta )\right) \Delta u(z) \\
&=&\int_{S_{c,u}}\left( \int_{B_{c,u}}P_{B_{c,u}}(z,\zeta )\Delta
u(z)\right)
\varphi (\zeta )\,d\nu _{c}(\zeta ) \\
&=&\int_{S_{c,u}}\varphi (\zeta )U_{c}(\zeta )\,d\nu _{c}(\zeta ).
\end{eqnarray*}
Hence $\mu _{c,u}=U_{c}\nu _{c}$. Another observation using Fubini's theorem gives
\begin{eqnarray*}
\int_{S_{c,u}}U_c (\zeta )d\nu _{c}(\zeta )
&=&\int_{S_{c,u}}\left ( \int_{B_{c,u}}P_{B_{c,u}}(z,\zeta)\Delta
u(z)   \right )d\nu _{c}(\zeta ) \\
&=&\int_{B_{c,u}}\left( \int_{S_{c,u}}P_{B_{c,u}}(z,\zeta
)d\nu _{c}(\zeta )\right) \Delta u(z) \\
&=&\int_{B_{c,u}} \Delta u(z) =\|\mu_{c,u}\|<\infty.
\end{eqnarray*}
Thus $U_c\in L^1(\nu_c)$. This completes the proof.
\end{proof}

The next results are restatements from \cite{PolSte08} and they establish basic observations on the
classes of Hardy spaces.

\bP\cite[Corollary 3.2]{PolSte08}\label{P:OrderHardySpaces} Let $u$
and $v$ be continuous subharmonic exhaustion functions on $G$ and
let $K$ be a compact set in $G$ such that $bv(z) \leq u(z)$ for some
constant $b
> 0$ and all $z \in G \backslash K$ . Then $sh_v\sbs sh_u$ and
$\|\varphi\|_u \leq b \|\varphi\|_v$ for every $\varphi\in sh_v$.
\eP

The following result is basically contained in the proof of \cite[Thorem 3.6]{PolSte08} taking $n=1$.
\bP\label{P:PSHNormDominatesSupNorm}
Let $v$ be a continuous subharmonic exhaustion function on $G$, $K\sbs G$
be compact and $V\sbs\sbs G$ be an open set containing $K$. Suppose that there exists
a constant $s>0$ so that  $v(z)\leq sg_G(z,w)$ for every $w\in K$ and $z\in G\backslash K$.
Then
\begin{eqnarray*}
\varphi (w)\leq \frac{s}{2\pi}\|\varphi\|_v,\,\,\,\,w\in K
\end{eqnarray*}
for every nonnegative $\varphi\in sh(G)$.
\eP
\section{Invariant subspaces of the shift operator}\label{Sec:InvSubspace}
We reserve this section for general results about the
Poletsky-Stessin-Hardy spaces in the disk. In view of Theorem
\ref{Th:InnerOuterFactorization} and combining with the results of
the classical Hardy space theory every function $f\in
H^p_u(\mathbb{D})$ is of the form $f=gh$, where $g$ is an inner
function, $h\in H^p_u(\mathbb{D})$ and $h$ is not equal to zero
anywhere in $\mathbb{D}$. The following factorization can be
compared with \cite[Theorem 17.10]{Rudin}.
\bT\label{Th:2pDescription} Suppose $0<p<\infty$, $f\in
H^p_u(\mathbb{D})$, $f\not \equiv 0$, and $B$ is the Blaschke
product formed with the zeros of $f$. Then there is a zero-free
function $h\in H^2_u$ such that
\begin{eqnarray}\label{Eq:2pDescription}
f=Bh^{2/p}.
\end{eqnarray}
In particular, every $f\in H^1_u$ is a product
\begin{eqnarray}\label{Eq:12Description}
f=gh,
\end{eqnarray}
in which both factors are in $H^2_u$. \eT

\begin{proof} By Theorem \ref{Th:InnerOuterFactorization} $f/B\in H^p_u$. There exists
$\varphi\in hol(\mathbb{D})$ so that
\[f/B=e^{\varphi}.\]Put
\[h=e^{p\varphi/2}\] and this function satisfies (\ref{Eq:2pDescription}). Also
$h\in H^2_u$ by Theorem \ref{T:CharacHardyClassesBoundaryVal}. Now
to obtain (\ref{Eq:12Description}) write (\ref{Eq:2pDescription}) in
the form $f=(Bh)h$. We just note that $Bh\in H^2_u$ from Theorem
\ref{Th:InnerOuterFactorization}.
\end{proof}

In this section we extend the classical characterization of
invariant subspaces of the multiplication operator $M_z$ by $z$ to
the Poletsky-Stessin-Hardy spaces. Our proof basically follows the
proof of Beurling's Theorem in \cite{Rudin}. In the next theorem the
function $V$ is the one defined by (\ref{Eq:DefnV}) and
$H^2(\mathbb{D})$ is the classical Hardy space.
\bT\label{Th:BeurlingThm} Let $Y\not =\{0\}$ be a closed
$M_z$-invariant subspace of $H^2_u(\mathbb{D})$. Then there exists a
function $\varphi \in H^2_u$ so that $|\varphi^{\ast}
(\zeta)|^2V(\zeta)=1$ for almost every $\zeta\in\partial\mathbb{D}$
and $Y=\varphi H^2(\mathbb{D})$. \eT

\begin{proof} There is a smallest integer $k$ so that $Y$ contains a function of the form
\begin{eqnarray*}
 f(z)=\sum_{n=k}^{\infty}a_nz^n,\,\,\,\,a_k=1.
\end{eqnarray*}
Hence, $f\not\in zY$ and $zY$ is a proper closed subspace of $Y$. There exists then
a function $\varphi\in Y$ with $\|\varphi\|_{2,u}=1$ so that $\varphi\perp zY$. In particular, $\varphi\perp z^n\varphi$
for every $n\geq 1$ which means that
\begin{eqnarray*}
 \frac{1}{2\pi}\int_0^{2\pi} |\varphi ^{\ast}(e^{i\theta})|^2V(e^{i\theta})e^{in\theta}\,d\theta =0
\end{eqnarray*}
for every integer $n\not =0$. From the Fourier coefficients we see that the function
$|\varphi ^{\ast}|^2V=1$ almost everywhere on $\partial\mathbb{D}$. Since $Y$ is $M_z$-invariant,
all functions $\varphi g$, where $g$ is a poynomial, is in $Y$. Now note that
\begin{eqnarray}\label{Eq:H2NormEqualsMultNorm}
 \|\varphi g\|^2_{2,u}=\int _{\partial \mathbb{D}}|\varphi^{\ast}(\zeta)|^2|g(\zeta)|^2V(\zeta)\,d\nu(\zeta)=\|g\|^2_2
\end{eqnarray}
when $g$ is a polynomial. The polynomials are dense in $H^2$, and
since $Y$ is closed, it follows that $\varphi H^2\sbs Y$. We have to
prove that this inclusion is in fact an equality. First notice that
the set $\varphi H^2$ is a closed subspace of $Y$. To see this we
take $\varphi g_n\in \varphi H^2$ converging to a function $h\in Y$.
Considering the norm in $H^2_u$ this gives that the sequence
$\{g_n\}$ is Cauchy in $H^2$, hence the functions $g_n$ converge in
$H^2$ to a $g\in H^2$. By (\ref{Eq:H2NormEqualsMultNorm}) it follows
that $\varphi g_n$ converges to $\varphi g$ in $H^2_u$, thereofore,
$h=\varphi g$. Now let $\psi\in Y$ and $\psi\perp \varphi H^2$. We
need to show that $\psi =0$. If $\psi\perp \varphi H^2$, then
$\psi\perp z^n\varphi$ for $n=0,1,2,\ldots$. On the other hand
$z^n\psi\in zY$ for $n\geq 1$ and this shows that $z^n\psi\perp
\varphi$. Thus we see that all Fourier coefficients of
$\psi^{\ast}\ovr {\varphi^{\ast}}V$ are $0$, hence $\psi^{\ast}\ovr
{\varphi^{\ast}}V=0$ almost everywhere on $\partial \mathbb{D}$. The
conclusion from here is that $\psi^{\ast}=0$ almost everywhere on
$\partial \mathbb{D}$, hence $\psi=0$, and this proves the claim
that $Y=\varphi H^2(\mathbb{D})$.
\end{proof}
Following the Theorem above we will call a function $\varphi\in H^2_u$ a $u$-inner
function if $|\varphi^{\ast} (\zeta)|^2V(\zeta)$ equals $1$ for almost every $\zeta\in\partial\mathbb{D}$.
If, moreover, $\varphi (z)$ is zero-free, we will say that $\varphi$ is a singular $u$-inner function.
Since $H^2_u$ is itself a closed invariant subspace of $M_z$, we obtain the following result
from Theorem \ref{Th:BeurlingThm}.
\bC\label{Cor:PSHSpacesAreMultipleOfHardy}
Let $u$ be a continuous subharmonic exhaustion function on $\mathbb{D}$ so that
$H^2_u\not =\{0\}$. If $\psi\in H^2$ is an inner function, then there exists a $u$-inner function
$\varphi \in H^2_u$ so that $\psi H^2_u=\varphi H^2$ and these spaces are isometric. In particular, there exists a
$u$-inner function $\varphi \in H^2_u$ so that $H^2_u=\varphi H^2$ and these spaces are isometric.
\eC

Corollary \ref{Cor:PSHSpacesAreMultipleOfHardy} gives another characterization
of the Poletsky-Stessin-Hardy spaces. Combining with the statement of Theorem \ref{Th:2pDescription}
we extend this result to the case $p\not =2$.
\bT
Suppose $0<p<\infty$, $f\in H^p_u(\mathbb{D})$, $f\not \equiv 0$, and $B$ is the Blaschke product formed with the
zeros of $f$. Then there are zero-free $\varphi\in H^2_u$ and $h\in H^2$ so that $\varphi$ is $u$-inner and
\begin{eqnarray}
 f=B\varphi ^{2/p}h^{2/p}.
\end{eqnarray}
Moreover, $\|f\|_{p,u}=\|h^{2/p}\|_p$.
\eT


\section{Examples}\label{Sec:Examples}

In this section we provide non-trivial examples of
Poletsky-Stessin-Hardy spaces. One way of obtaining a negative
subharmonic exhaustion function $u$ on the unit disk is the
following. Take any subharmonic function $v$ in disk so that $v(z)$
is continuous as $z$ approaches the boundary. From the
Poisson-Jensen formula $v$ is the sum of a Poisson integral
$P_{\mathbb{D}}v$ and a Green integral $G_{\mathbb{D}}v$. Let
$u=G_{\mathbb{D}}v$. Then $u$ is a negative subharmonic exhaustion
function on $\mathbb{D}$. For any continuous function $\varphi$ on
$\partial \mathbb{D}$  we set
\[ P_{\mathbb{D}}\varphi (z):=\frac{1}{2\pi }\int P(r,\theta -t)\varphi
(e^{it})dt,\,\,\,\,z=re^{i\theta },
\]
the Poisson integral of $\varphi$ on $\mathbb{D}$.

\begin{Example}\label{Ex: ExhFuncDefn}
Let $\Delta :=\{z=x+iy\in \mathbb{C}:(x-1/2)^{2}+y^{2}<1/4\}$. Let $%
\varphi _m(z):=-(1-x)^m$ for any $z\in \mathbb{D}$. Then $\Delta
\varphi_m=m(1-m)(1-x)^{m-2}$. Hence $\varphi_m$ is subharmonic in
$\mathbb{D}$ when $0<m\leq 1$. Let
\[G_{\mathbb{D}}\varphi_m(z):=\frac{1}{2\pi}\int_{\mathbb{D}}\log\left | \frac{z-w}{1-z\ovr w}  \right |\,\Delta \varphi_m(w)\] for any
$z\in \mathbb{D}$ be the Green integral of $\varphi_m$. Let $\Om
:=\mathbb{D}\backslash \ovr{\Delta}$ and $\sigma _m$ be the
restriction of $\Delta \varphi_m$ to $\Delta$. Define
\[G_{\mathbb{D}}\sigma_m(z):=\frac{1}{2\pi}\int_{\mathbb{D}}\log\left | \frac{z-w}{1-z\ovr w}  \right |\,d \sigma_m(w),\,\,\,\,z\in\mathbb{D}.\]
We denote by $\widetilde{\varphi _{m}}$ the continuous function in
the boundary of $\Om$ with values equal to zero on $\partial
\mathbb{D}$ and equal to $\varphi _m$ on the boundary of $\Delta$.
Let us denote the harmonic function in $\Om$ with boundary values
equal to $\widetilde{\varphi _{m}}$ by $P_{\Omega
}(\widetilde{\varphi _{m}} )(z)$. We set
\[
v_{m}(z):=\left\{
\begin{array}{cc}
\varphi _{m}(z)\text{ } & \text{if \ }z\in \overline{\Delta }, \\
P_{\Omega }(\widetilde{\varphi _{m}} )(z) & \text{if \ }z\in \overline{\Omega }.%
\end{array}%
\right.
\] Then $v_m$ is continuous and subharmonic in $\mathbb{D}$.

One can calculate that
\begin{eqnarray*}\int_{\mathbb{D}}d\sigma_m &=& m(1-m)\int_{0}^{1}\int_{-\sqrt{x(1-x)}}^{\sqrt{x(1-x)}}(1-x)^{m-2}\,%
dydx \\ \nonumber &=& 2m(1-m)\int_{0}^{1}\sqrt{x}(1-x)^{m-3/2}
\,dx.\end{eqnarray*}Hence $\int_{\mathbb{D}}d\sigma_m<\infty$
precisely when $1/2<m\leq 1$.
\end{Example}

\emph{Throughout this section we will use the notation of Example
\ref{Ex: ExhFuncDefn}. We will take $0<m\leq 1$.}\newline

\begin{Proposition} \label{Pr:IneqExhFunc} The functions
$v_m$, $G_{\mathbb{D}}\varphi _{m}$ and $G_{\mathbb{D}}\sigma _{m}$
are continuous subharmonic exhaustions on $\mathbb{D}$ and we have
\[
\varphi _{m}(z)\leq G_{\mathbb{D}}\varphi _{m}(z)\leq
G_{\mathbb{D}}\sigma _{m}(z)
\]  and
\[\varphi _{m}(z)\leq v_{m}(z)\leq G_{\mathbb{D}}\sigma _{m}(z)
\] for every $z\in \mathbb{D}$.
\end{Proposition}

\begin{proof} $\Delta v_m\geq \sigma_m$ in
$\mathbb{D}$. Hence $v_m-G_{\mathbb{D}}\sigma _{m}$ is subharmonic
in $\mathbb{D}$. Also $v_m=G_{\mathbb{D}}\sigma _{m}=0$ in $\partial
\mathbb{D}$. Thus $v_m\leq G_{\mathbb{D}}\sigma _{m}$. The other
inequalities are clear.
\end{proof}

From Proposition \ref{P:OrderHardySpaces} and Proposition
\ref{Pr:IneqExhFunc} we get the following.

\begin{Corollary} Let $u_m=G_{\mathbb{D}}\sigma _{m}$. Then $H_{v_{m}}^{p}\subset
H_{u_m}^{p}$.
\end{Corollary}

As a result of Proposition \ref{Pr:IneqExhFunc} we observe that
\[\overline{B}_{c,G_{\mathbb{D}}\sigma _{m}}\cap\overline{\Delta}
\sbs\{z=x+iy\in \overline{%
\Delta }:0\leq x\leq 1-(-c)^{1/m},\,y^{2}\leq x(1-x)\}.
\]

In the next example we construct an exhaustion function $u$ in the
unit disk so that $H^1_u$ is not trivial and is a proper subset of
$H^1$.

\begin{Example}\label{Ex:HardyInfiniteMass}
Let $0<m\leq 1/2$, $p>0$, and $u=u_m=G_{\mathbb{D}}\sigma _{m}$ in
Example \ref {Ex: ExhFuncDefn}. Let $K_{c,u}:=\overline{B}_{c,u}\cap
\overline{\Delta }$. We know that $\int \Delta u=\infty$. Hence,
$1\not\in H_{u}^{p}(\mathbb{D})$. \
\newline
\newline
Now let us show that any function $f\in $hol$(\mathbb{D})$ with $\Delta
(|f|^p)\leq M$ on $\mathbb{D}$ and $|f|^p\leq (1-x)^{1/2}$ on $\Delta $ belongs
to $H_{u}^{p}(\mathbb{D})$. In fact, from (\ref{Eq:DemJenLel}) one can show
that
\begin{align*}
\mu _{c,u}(|f|^p)& \leq \int_{K_{c,u}}|f|^p\Delta u+C\leq 2m(1-m)%
\int_{0}^{1}x^{1/2}(1-x)^{m-1}\,dx+C \\
& \leq 2m(1-m)\int_{0}^{1}(1-x)^{m-1}\,dx+C=2(1-m)+C,
\end{align*}%
where $C=-\int_{\mathbb{D}}u\Delta (|f|^p)\geq 0$ is a number
independent of $c$. Thus, $f\in H_{u}^{p}(\mathbb{D})$.\newline
\newline
To finish the example, let us take $f(z)=[a(1-z)]^{2/p}$, where we
take the principal branch of the logarithm and $|a|\leq 1/2$.
Clearly $f$ and $\Delta (|f|^p)$ are bounded on $\mathbb{D}$ and one
can see
that $|f|^p\leq (1-x)^{1/2}$ on $\Delta $. Therefore, $\{0\}\not=H_{u}^{p}(%
\mathbb{D})\subsetneq H^{p}(\mathbb{D})$.
\end{Example}

The next example is similar to the one above. In this case $u$ has finite Laplace mass.
\begin{Example}\label{Ex:HardyFiniteMass}
 Now let $1/2<m<1$ with $u=u_m=G_{\mathbb{D}}\sigma _{m}$ in Example \ref
{Ex: ExhFuncDefn}. It follows from the same discussion in Example \ref{Ex:HardyInfiniteMass}
that the function $f(z)=(1-z)^{1/p}$ belongs to $H^p_u$. Take $m-1/2<t<1/2$ and set
 \begin{eqnarray}
  g(z):=\frac{1}{(1-z)^{2t/p}}.
 \end{eqnarray}
 We will show that the function $g$ belongs to $H^p$, but not in $H^p_u$.
 Indeed, we have
 \begin{eqnarray*}
  |g(re^{i\theta})|^p=\frac{1}{(\sin^2\theta+(\cos\theta -r)^2)^t}\leq \sin^{-2t}\theta.
 \end{eqnarray*}
Now $\int_0^{2\pi}\sin^{-2t}\theta d\theta<\infty$ exactly when $t<1/2$. Hence $g\in H^p$.
Considering the $H^p_u$ norm
\begin{eqnarray}
 \nonumber \|g\|_{u,p} &\geq & \int _{\Delta}|g|^p \Delta u=
 \int_0^1\int_{-\sqrt{x(1-x)}}^{\sqrt{x(1-x)}}\frac{(1-x)^{m-2}}{[(1-x)^2+y^2]^t}dydx\\
 \label{Eq:NormEstFiniteMass} &\geq & 2^{-t+1}\int_0^1 \sqrt{x}(1-x)^{m-3/2-t}dx.
\end{eqnarray}
The integral in \eqref{Eq:NormEstFiniteMass} is infinite when $t>m-1/2$. Therefore
$g$ does not belong to $H^p_u$. Again $\{0\}\not=H_{u}^{p}(%
\mathbb{D})\subsetneq H^{p}(\mathbb{D})$.
\end{Example}


\section{Further remarks}
\par 1. Theorem \ref{T:CharacHardyClassesBoundaryVal} together with Chern-Nirenberg inequality
shows that if $\partial G$ has
$C^2$ boundary and $\Delta u\in L^{\infty}(G)$,
then $H^p_u=H^p_{\lm}$, the Stein-Hardy space. \\
\par 2. If $u$ is radial, that is to say, if $u(z)=u(|z|)$ on the unit disk, then there are two possibilities:
Either $H^p_u=\{0\}$ which is the case when $\int \Delta u=\infty$, or $H^p_u=H^p$.
Note that if $u$ is radial, then so is the measure $\Delta u$ and
\[V(\zeta)=\int_{\mathbb{D}}P_{\mathbb{D}}(z,\zeta)\Delta u(z)=\int_0^1r\Delta u(r)dr,\]
which is either equal to $\infty$ for every $\zeta\in\partial \mathbb{D}$ or it is constant.\\



\end{document}